\documentclass[a4paper,11pt]{amsart}
\usepackage{lmodern}
\usepackage{amsmath,amsfonts,amssymb,mathrsfs,charter,amsthm, multicol,color, hyperref}

\usepackage[inner=1.1in,outer=1.1in,height=9.3in]{geometry}

\newtheorem{theorem}[subsection]{Theorem}
\newtheorem{lemma}[subsection]{Lemma}
\newtheorem{cor}[subsection]{Corollary}

\numberwithin{equation}{section}
\newtheorem{prop}[subsection]{Proposition}

\newcommand{\dist}{{\mathop{\rm dist}}}
\newcommand{\ext}{{\mathop{\rm ext}}}
\newcommand{\conv}{{\mathop{\rm \overline{conv}}}}

\newcommand{\A}{{\mathcal A}}
\newcommand{\B}{{\mathcal B}}
\newcommand{\C}{{C}}

\newcommand{\hc}{{\mathcal H}}
\newcommand{\hk}{{\mathcal K}}

\newcommand{\eps}{\varepsilon}
\newcommand{\norm}[1]{\left\lVert#1\right\rVert}

\newcommand{\bF}{{\mathbb{F}}}
\newcommand{\bR}{{\mathbb{R}}}
\newcommand{\bC}{{\mathbb{C}}}
\newcommand{\bN}{{\mathbb{N}}}

\newcommand{\tr}{{\mathrm{tr }}}

\newcommand{\bra}{\langle}
\newcommand{\ket}{\rangle}
\newcommand{\tens}[1]{\mathbin{\mathop{\otimes}\limits_{#1}}}
\begin{document}


\title{Gateaux derivative of $C^*$ norm}
\author[S. Singla]{Sushil Singla}
\address{Department of Mathematics, Shiv Nadar University, NH-91, Tehsil Dadri, Gautam Buddha Nagar, U.P. 201314, India.}
\email{ss774@snu.edu.in}
\subjclass[2010]{Primary 49J50, 47N10, 46L30, 46L08, 46B20; Secondary 46L05, 41A52, 46M10} 
\keywords{Trace Class, Subdifferential Set, Gateaux differentiablity, smooth points, states, Birkhoff-James orthogonality, Hilbert $C^*$-modules}
\maketitle

\begin{abstract} 
We find an expression for Gateaux derivative of the $C^*$-algebra norm. This gives us alternative proofs or generalizations of various known results on the closely related notions of subdifferential sets, smooth points and Birkhoff-James orthogonality for spaces $\mathscr B(\hc)$ and $C_b(\Omega)$. We also obtain an expression for subdifferential sets of the norm function at $A\in\mathscr B(\hc)$ and a characterization of orthogonality of an operator $A\in\mathscr B(\hc, \hk)$ to a subspace, under the condition $\dist(A, \mathscr K(\hc))< \norm{A}$ and $\dist(A, \mathscr K(\hc, \hk))< \norm{A}$ respectively.  
\end{abstract}

\setlength{\parindent}{0pt}
\setlength{\parskip}{1.6ex}

\pagestyle{headings}

\section{Introduction}\label{intro}

Let $\bF$ stand for $\bR$ or $\bC$. Let $\A$ be a $C^*$-algebra over $\bF$. Let $\C_b(\Omega)$ be a $C^*$-algebra of bounded $\bF$-valued functions on a locally compact Hausdorff space $\Omega$ and $\C_0(\Omega)\subseteq \C_b(\Omega)$ be the space of functions vanishing at infinity.  It is well known that every commutative $C^*$-algebra is isomorphic to $\C_0(\Omega)$. Let $\hc$ and  $\hk$ be Hilbert spaces over $\bF$ with the inner product $\bra\cdot|\cdot\ket$. Let $\mathscr B(\hc, \hk)$ and $\mathscr K(\hc, \hk)$  be the normed spaces of bounded and compact $\bF$-linear operators from $\hc$ to $\hk$ with operator norm, respectively.  Let $\mathscr B(\hc)$ and $\mathscr K(\hc)$ stands for $\mathscr B(\hc, \hc)$ and $\mathscr K(\hc, \hc)$, respectively. It is also well known that any $C^*$-algebra is isomorphic to a $C^*$-subalgebra of $\mathscr B(\hc)$. 

Let $(V,\|\cdot\|)$ be a normed space over $\bF$. Let $v\in V$ and let $W$ be a subspace of $V$. Let $\dist(v,W)$ denotes $\inf\{\|v-w\|:w\in W\}.$ The \emph{Gateaux derivative of $\|\cdot\|$ at $v$ in direction of $u$} is defined as $\lim\limits_{t\rightarrow 0} \dfrac{\|v+tu\| - \|v\|}{t}$. We say $\|\cdot\|$ is \emph{Gateaux differentiable} at $0\neq v$ if and only if for all $u\in V$, $\lim\limits_{t\rightarrow 0} \dfrac{\norm{v+tu}-\norm{v}}{t}$ exists. A concept related to the Gateaux derivative of norm function is the \emph{subdifferential set of norm function} (see \cite{hiriart}). The subdifferential set of  norm function at $v\in V$, denoted by $\partial\|v\|$,  is the set of bounded linear functionals $f \in V^*$ satisfying $\|u\|-\|v\|\geq \text{Re } f(u-v) \quad \text{for all } u\in V.$ It is easy to prove that for $0\neq u, v\in V$, $\partial \|v\|=\{f\in V^*: f(v)=\|v\|, \|f\| = 1\}$ and \begin{equation}\lim\limits_{t\rightarrow 0^+} \dfrac{\norm{v+tu} - \norm{v}}{t} = \max\{\text{Re }f(u) : f\in \partial \|v\|\}. \label{67}\end{equation}

There is a notion to define orthogonality in a normed space, called \emph{Birkhoff-James orthogonality} which is closely related to the Gateaux derivative of norm function (see \cite{James}).  We say $v$ is said to be Birkhoff-James orthogonal to $W$ if $\|v\|\leq \|v+w\|$ for all $w \in W$. Let $u, v\in V$, then $v$ is said to be Birkhoff-James orthogonal to $u$ if and only if $v$ is Birkhoff-James orthogonal to $\bF u$. In \cite{Keckic 2}, the concept of \emph{$\varphi$-Gateaux derivative} was introduced which also gives a characterization of Birkhoff-James orthogonality. The number $D_{\varphi, v}(u)= \lim\limits_{t\rightarrow 0^+}\dfrac{\norm{v+te^{\iota\varphi}u}-\norm{v}}{t}$ is called the $\varphi$-Gateaux derivative of $\|\cdot\|$ at $v$, in the $u$ and $\varphi$ directions. 

\begin{prop}\label{22}\cite[Proposition 1.5]{Keckic} The limit $D_{\varphi, v}(u)$ always exists for $u, v\in V$ and $\varphi\in [0, 2\pi)$. And $v$ is Birkhoff-James orthogonal to $u$ if and only if $\inf_{\varphi}D_{\varphi, v}(u) \geq 0$.
\end{prop} 
 
In \cite{Keckic}, an expression for the $\varphi$-Gateaux derivative of $\mathscr B(\hc)$ norm was obtained and as an application, a proof of a known characterization of smooth points and Birkhoff-James orthogonality of an operator to another operator in $\mathscr B(\hc)$ was given. In this paper, an expression for $\varphi$-Gateaux derivative of the norm of $\A$ in terms of \emph{states} is obtained and various applications are discussed. Few definitions are in order. A \emph{positive element} in $\A$ is of the form $a^*a$ for some $a\in\A$. A \emph{positive functional} on $\A$ is a linear functional which takes positive elements of $\A$ to non-negative real numbers.  For $\bF=\bC$, a \emph{state} $\psi$ on $\A$ is positive functional of norm one. For $\bF=\bR$, an additional requirement for $\psi$ to be a state is that $\psi(a^*) = \psi(a)$ for all $a\in\A$.  Let $\mathcal S_{\A}$ be the set of states on $\A$. The main theorem of this paper is as below.

\begin{theorem}\label{51} Let $0\neq a, b\in\A$. Then $$\lim\limits_{t\rightarrow 0^+} \dfrac{\norm{a+t b} - \norm{a}}{t} = \dfrac{1}{\norm{a}}\max\{\text{Re }\psi(a^*b) : \psi\in\mathcal S_{\A}, \psi(a^*a) = \norm{a}^2\}.$$\end{theorem}

It is well known that if $\mathcal I$ is a two-sided closed ideal of $\A$, then any $\psi\in\mathcal S_{\mathcal I}$ has a unique extension $\tilde{\psi}\in\mathcal S_{\A}$. We will be using the same notation $\psi$ for the extension $\tilde{\psi}$ throughout the article. As an application of the above theorem, we get the following corollary.

\begin{cor}\label{68} Let $\mathcal I$ be a two-sided closed ideal of $\A$. Let $0\neq a\in\A$ such that $\dist(a, \mathcal I)< \norm{a}$. Then for any $0\neq b\in\A$, we have $$\lim\limits_{t\rightarrow 0^+} \dfrac{\norm{a+t b} - \norm{a}}{t} = \dfrac{1}{\norm{a}}\max\{\text{Re }\psi(a^*b) : \psi\in\mathcal S_{\mathcal I}, \norm{\psi} = 1, \psi(a^*a) = \norm{a}^2\}.$$
\end{cor}

Another concept which is closely related to the Gateaux derivative of norm function and Birkhoff-James orthogonality is that of the  \emph{smooth points} of the unit ball in a normed space. We say that a vector $v$ of norm one is a smooth point of the unit ball of $V$ if there exists a unique functional $F_v$, called the support functional, such that $\norm{F_v} = 1$ and $F_v(v) = 1$. It is a general fact that $v$ is a smooth point of the unit ball of $v$ if and only if the norm is Gateaux differentiable at $v$. And in that case, $\lim\limits_{t\rightarrow 0} \dfrac{\norm{v+tu}-\norm{v}}{t} = \text{ Re }F_v(u)$. Further, $v$ is Birkhoff-James orthogonal to $u$ if and only if $F_v(u) = 0$ (see \cite[Theorem 2.1]{Abatzoglou} and \cite[Proposition 1.3]{Keckic}). Along the lines of proof of Corollary 2.2 of \cite{Keckic 1}, we get that the smooth points of the unit ball of $\C_0(\Omega)$ are those $f\in\C_0(\Omega)$ such that $\norm{f}_{\infty} = 1$  such that $|f|$ is equal to $1$ at exactly one point. For a characterization of smooth points of  the unit ball of $C_b(\Omega)$ for $\Omega$ a normal space, see \cite[Corollary 3.2]{Keckic 1}. 

In Section \ref{proofs}, we give the proof of Theorem \ref{51} and Corollary \ref{68}. In Section \ref{applications}, we obtain different proofs for several known results and their generalizations. In Corollary \ref{52}, a proof for a characterization of Birkhoff-James orthogonality of an element to a subspace in a $C^*$-algebra is obtained. In Corollary \ref{71}, we obtain an expression for Gateaux derivative of norm function of $\mathscr B(\hc)$ at $A$ in direction of $B$ when $\dist(A, \mathscr K(\hc))< \norm{A}$. Using Corollary \ref{71}, an alternate proof for a known characterization of smooth points of the unit ball of $\mathscr B(\hc)$ is given in Corollary \ref{78}. In Corollary \ref{999}, we obtain the subdifferential set $\partial \|A\|$, when $\dist(A, \mathscr K(\hc))< \norm{A}$. 

\section{Proofs}\label{proofs}

To prove theorem \ref{51}, we need the following lemma.

\begin{lemma}\label{23} Let $a,b\in\A$. Then 
\begin{equation*}\lim\limits_{t\rightarrow 0^+} \dfrac{\norm{a+t b} - \norm{a}}{t} = \dfrac{1}{\norm{a}}\lim\limits_{t\rightarrow 0^+} \dfrac{\norm{a^*a+t a^*b} - \norm{a^*a}}{t}.\end{equation*}
\end{lemma}
\begin{proof} We have $$\dfrac{1}{\norm{a}}\lim\limits_{t\rightarrow 0^+} \dfrac{\norm{a^*a+t a^*b} - \norm{a^*a}}{t}\leq\dfrac{1}{\norm{a}}\lim\limits_{t\rightarrow 0^+} \dfrac{\norm{a^*}(\norm{a+t b} - \norm{a})}{t} = \lim\limits_{t\rightarrow 0^+} \dfrac{\norm{a+t b} - \norm{a}}{t}.$$ Now we need to prove $\lim\limits_{t\rightarrow 0^+} \dfrac{\norm{a+t b} - \norm{a}}{t} \leq \dfrac{1}{\norm{a}}\lim\limits_{t\rightarrow 0^+} \dfrac{\norm{a^*a+t a^*b} - \norm{a^*a}}{t}.$\begin{eqnarray*}\lim\limits_{t\rightarrow 0^+} \dfrac{\norm{a+t b} - \norm{a}}{t} &=&  \lim\limits_{t\rightarrow 0^+} \dfrac{\norm{a+t b}^2 - \norm{a}^2}{t(\norm{a+t b} + \norm{a})}\\
&=&\dfrac{1}{2\norm{a}}\lim\limits_{t\rightarrow 0^+} \dfrac{\norm{a^*a+t(a^*b+b^*a) + t^2b^*b} - \norm{a^*a}}{t}\\
&=&\dfrac{1}{2\norm{a}}\lim\limits_{t\rightarrow 0^+} \dfrac{\norm{a^*a+t(a^*b+b^*a)} - \norm{a^*a}}{t}\\
&\leq&\dfrac{1}{2\norm{a}}\lim\limits_{t\rightarrow 0^+} \dfrac{\norm{a^*a+2ta^*b} + \norm{a^*a+2tb^*a)} - 2\norm{a^*a}}{2t}\\
&=&\dfrac{1}{\norm{a}}\lim\limits_{t\rightarrow 0^+} \dfrac{\norm{a^*a+2ta^*b} - \norm{a^*a}}{2t}\\
&=&\dfrac{1}{\norm{a}}\lim\limits_{t\rightarrow 0^+} \dfrac{\norm{a^*a+t a^*b} - \norm{a^*a}}{t}.
\end{eqnarray*}\end{proof}

\begin{lemma}\label{8989} Let $f\in\A^*$ of norm one such that $f(a^*a) = \norm{a^*a}$ for some $a\neq 0$. Then $f\in \mathcal S(\A)$.
\end{lemma}
\begin{proof}  We extend $f$ to a functional $\tilde{f}$ of norm one on $\A^+$, where $\A^+$ is unitization of $\A$. Using Theorem II.6.3.4 of \cite{blackadar}, we have $\tilde{f} = (f_1 -f_2) +\iota (f_3-f_4)$ where $f_i$ is positive linear functional with $\norm{f_1-f_2} = \norm{f_1}+\norm{f_2}$ and $\norm{f_3-f_4} = \norm{f_3}+ \norm{f_4}$. Now $\tilde{f}(a^*a) = \norm{a^*a}$ implies $f_1(a^*a) = \norm{a^*a}$ and $\norm{f_1} =1$. This gives $f_1(1_{\A}) = 1$ and thus $\tilde{f}(1_{\A}) =1$. And using Theorem II.6.2.5(ii), $\tilde{f}$ is a positive functional, hence $f$ is a positive functional and thus a state.
\end{proof}

{\it Proof of Theorem \ref{51}} We have $\partial\norm{a^*a} = \{f\in\A^* : \norm{f} = 1, f(a^*a) = \norm{a^*a}\}$. Now any linear functional $f$ of norm one such that $f(a^*a) = \norm{a^*a}$ for some $a\neq 0$ is a state by Lemma \ref{8989}.
Hence we get $\partial\norm{a^*a} = \{\psi\in\mathcal S_{\A} :  \psi(a^*a) = \norm{a^*a}\}$. Now using Lemma \ref{23} and equation \eqref{67}, we get \begin{eqnarray*}\lim\limits_{t\rightarrow 0^+} \dfrac{\norm{a+t b} - \norm{a}}{t} &=& \dfrac{1}{\norm{a}}\lim\limits_{t\rightarrow 0^+} \dfrac{\norm{a^*a+t a^*b} - \norm{a^*a}}{t} = \dfrac{1}{\norm{a}}\max\{\text{Re }\psi(a^*b) : \psi\in \partial \|a^*a\| \}\\
&=&  \dfrac{1}{\norm{a}}\max\{\text{Re }\psi(a^*b) : \psi\in\mathcal S_{\A}, \psi(a^*a) = \norm{a^*a}\}.\end{eqnarray*} \qed

Before proving Corollary \ref{68}, we note that Theorem 3.1 of \cite{2019} follows as a corollary of Lemma \ref{23} and Proposition \ref{22}.

\begin{cor}\label{88}\cite[Theorem 3.1]{2019} Let $a\in\A$. Let $\B$ be a subspace of  $\A$. Then $a$ is Birkhoff-James orthogonal to $\B$ if and only if $a^*a$ is Birkhoff-James orthogonal to $a^*\B$.
\end{cor}

We will denote the closed convex hull of a set $A$ by $\conv(A)$. And for a convex set $C$, $\ext(C)$ will denote the set of extreme points of $C$. We now prove Corollary \ref{68} by imitating the proof of Lemma 3.1 of \cite{Wojcik}.

{\it Proof of Corollary \ref{68}}  We prove that if $\psi\in \A^*$ such that $\psi(a^*a) = \norm{a}^2$, then $\norm{\psi|_{\mathcal I}} = \norm{\psi} = 1$. The result then follows from Theorem \ref{51}. Now any two-sided ideal is an $M$-ideal in $\A$, by Theorem V.4.4 of \cite{Ideal}. So $\A^* = \mathcal I^\#\oplus_1\mathcal I^{\perp}$, by Proposition I.1.12 of \cite{Ideal}, where $\mathcal I^\# = \{f\in \A^* : \norm{f} = \norm{f|_{\mathcal I}}\}$. Now let $K = \{f\in \A^* : \norm{f} =1, f(a^*a) = \norm{a}^2\}$. Then $K$ is a weak*-compact convex subset of $\A^*$. By Krein-Milman theorem, $K = \conv(\ext(K))$. Since $K$ is a face of $\A^*$, $\ext(K)\subseteq \ext(\mathcal I^\#\oplus_1\mathcal I^{\perp}) = \ext(\mathcal I^\#)\cup\ext(\mathcal I^{\perp})$. If we prove $\ext(K)\cap\mathcal I^{\perp} = \emptyset$, then $\ext(K)\subseteq \ext(\mathcal I^\#)$, and hence $K\subseteq \mathcal I^\#$. Let $f\in\mathcal I^{\perp}$ and $\norm{f} = 1$. We show $f\notin K$ by proving $f(a^*a)\neq \norm{a}^2$. As an application of Corollary \ref{88}, we get that $\dist(a, \mathcal I)< \norm{a}$ implies $\dist(a^*a, \mathcal I)< \norm{a}^2$. Now $\dist(a^*a, \mathcal I)< \norm{a}^2$ implies there exists $b\in\mathcal I$ such that $\norm{a^*a-b} < \norm{a}^2$. Now $f(a^*a) = f(a^*a-b) \leq\norm{a^*a-b} < \norm{a}^2$.\qed

\section{Applications}\label{applications}

Using the Hahn-Banach theorem with Corollary \ref{88}, we get the following generalization of Theorem 3.1 of \cite{2019}.

\begin{cor}\label{52} Let $\mathcal I$ be a two-sided closed ideal of $\A$. Let $a\in \mathcal \A$ such that $\dist(a, \mathcal I)< \norm{a}$. Let $\B$ be a subspace of $\A$. Then $a$ is Birkhoff-James orthogonal to $\B$ if and only if there exists $\psi\in \mathcal S_{\mathcal I}$ such that $\psi(a^*a) = \norm{a}^2$ with $\psi(a^*b) = 0$ for all $b\in \mathcal I$.\end{cor}

Using the Riesz Representation Theorem with Corollary \ref{52}, we get the following  generalization of Corollary 3.2 of \cite{2019}.

\begin{cor}  Let $f\in\C_b(\Omega)$ such that $\dist(f, \C_0(\Omega)) < \norm{f}_{\infty}$. Let $\B$ be a subspace of $\C_b(\Omega)$. Then $f$ is Birkhoff-James orthogonal to $\B$  if and only if there exists a regular Borel probability measure $\mu$ on $\Omega$ such that {support of $\mu$ is contained in $\{x\in \Omega: |f(x)| = \norm{f}_\infty\}$} and $\int\limits_X \overline{f}h d\mu = 0 \text{ for all }h\in\B.$ If $|f|$ attains its norm at only one point $x_0$, then $f$ is orthogonal to $\B$ if and only if $h(x_0) = 0$ for all $h\in\B$.\end{cor} 

Next we obtain a characterization of orthogonality in $\mathscr B(\hc, \hk)$ by using Corollary \ref{52}. To do so, we require the following lemma, proved in \cite{2019}. For $u, v\in\hc$, $u\bar{\tens{}} v$ denotes the finite rank operator of rank one on $\hc$ defined as $u\bar{\tens{}} v(w) = \bra v | w\ket u$ for all $w\in\hc$. Let $\mathscr C_1(\hc)$ denotes the space of trace class on $\hc$, with the trace norm $\norm{\cdot}_1$. 

\begin{lemma}\label{99} (\cite[Lemma 3.8]{2019}) Let $A\in\mathscr B(\hc)$. Let $T\in\mathscr C_1(\hc)$ be a positive element of $\mathscr B(\hc)$ with $\norm{T}_1 =1$ and $\tr(AT) = \norm{A}$. Then there is an atmost countable index set $\mathcal J$, a set of positive numbers $\{s_j : j\in\mathcal J\}$ and an orthonormal set $\{u_j: j\in\mathcal J\} \subseteq \text{Ker}(T)^{\bot}$  such that \begin{enumerate}
\item[(i)]  $\sum\limits_{j\in\mathcal J} s_j =1$ ,
\item[(ii)]$Au_j = \norm{A}u_j$ for each $j\in \mathcal J$, \\ and
\item[(iii)] $T= \sum\limits_{j\in\mathcal J} s_j u_j\bar{\tens{}} u_j$.
\end{enumerate}\end{lemma} 

We know that $\mathscr K(\hc)^*$ is isomorphic to $\mathscr C_1(\hc)$ and an isomorphism is given by $\mathcal F: \mathscr C_1(\hc)\rightarrow \mathscr K(\hc)^*$ by $\mathcal F(T)(B) = \tr(BT)$ for all $B\in\mathscr K(\hc)$. And $\mathcal F(T)\in \mathcal S_{K(\hc)}$ if and only if $T$ is a positive element of $\mathscr K(\hc)$ with $\norm{T}_1 = 1$. We will be using this identification throughout this paper.

\begin{cor}\label{72} Let $A\in \mathscr B(\hc, \hk)$ be such that $\dist(A, \mathscr K(\hc, \hk))< \norm{A}$. Let $\B\subseteq \mathscr B(\hc, \hk)$ be a subspace. Then $A$ is Birkhoff-James orthogonal to $\B$ if and only if  there exists at most countable set $\mathcal J$, a set of positive numbers $\{s_j: j\in\mathcal J\}$ and an orthonormal set $\{u_j\in\hc: j\in\mathcal J\}$ such that 
\begin{enumerate}
\item[(i)]  $\sum\limits_{j\in\mathcal J} s_j =1$ ,
\item[(ii)]$A^*Au_j = \norm{A}^2u_j$  for each $j\in \mathcal J$, \\ and
\item[(iii)] $\sum\limits_{j\in\mathcal J} s_j\bra Au_j | Bu_j\ket = 0$ for all $B\in\B$. 
\end{enumerate} \end{cor}
\begin{proof} We can embed $\mathscr B(\hc, \hk)$ into $\mathscr B(\hc\oplus\hk, \hc\oplus\hk)$ isometrically as $T\rightarrow\begin{bmatrix} 0 & 0 \\ T & 0 \end{bmatrix}$. Then using Corollary \ref{88} for $\mathscr B(\hc\oplus\hk, \hc\oplus\hk)$, we get $A$ is Birkhoff-James orthogonal to $\B$ if and only if $A^*A$ is Birkhoff-James orthogonal to $A^*\B$. Since $\dist(A, \mathscr K(\hc, \hk))< \norm{A}$, then $\dist(A^*A, \mathscr K(\hc))< \norm{A^*A}$. Hence using Corollary \ref{52}, we get that there exists $T\in\mathscr C_1(\hc)$ such that $\tr(A^*AT) = \norm{A^*A}$ and $\tr(A^*BT) = 0$ for all $B\in\B$. And now the result follows using Lemma \ref{99}.
\end{proof} 

Using Corollary \ref{72} and Hausdorff-Toeplitz theorem, we get the following generalizations of Theorem 1.1 of \cite{Bhatia}.

\begin{cor} Let $A\in \mathscr B(\hc, \hk)$ such that $\dist(A, \mathscr K(\hc, \hk))< \norm{A}$. Let $B\in \mathscr B(\hc, \hk)$. Then $A$ is Birkhoff-James orthogonal to $B$ if and only if  there exists a sequence of unit vectors $u_n \in\hc$ such that $\norm{Au_n} = \norm{A}$ for all $n\in\bN$ and $\lim\limits_{n\rightarrow\infty}\bra Au_n | Bu_n\ket = 0.$
\end{cor}

It is worth mentioning that $A\in \mathscr B(\hc)$ such that the set $H_0 = \{u\in \mathscr B(\hc) : \norm{Au} = \norm{A}\}$ is finite dimensional and $\norm{A}_{H_0^{\perp}} < \norm{A}$, then $\dist(A, \mathscr K(\hc))< \norm{A}$. And now Theorem 2.8 of \cite{arxiv} Theorem 3.1 of \cite{Sain} follows from Corollary \ref{72}.

In Theorem 2.6 of \cite{Keckic 2}, Keckic proved that if $\hc$ is a seperable Hilbert space and $A, B\in \mathscr K(\hc)$, then $\lim\limits_{t\rightarrow 0^+} \dfrac{\norm{A+t B} - \norm{A}}{t} = \max\limits_{\norm{u} = 1, |A|u = \norm{A}u} \mathop{Re }\bra u|U^*Bu\ket$, where $A= U|A|$ is polar decomposition. The next result gives the generalization of this, for any Hilbert space $\hc$ and for $A, B\in\mathscr B(\hc)$ with $\dist(A, \mathscr K(\hc))< \norm{A}$.

\begin{cor}\label{71} Let $0\neq A\in \mathscr B(\hc)$ be such that $\dist(A, \mathscr K(\hc))< \norm{A}$, then for any $0\neq B\in \mathscr B(\hc)$ \begin{align}\label{8000}\lim\limits_{t\rightarrow 0^+} \dfrac{\norm{A+t B} - \norm{A}}{t} &= \max\limits_{ \norm{u} = \norm{v} = 1, Au = \norm{A}v} \mathop{Re }\bra v|Bu\ket\nonumber\\
&=  \dfrac{1}{\norm{A}}\max\limits_{ \norm{u} = 1,  A^*Au = \norm{A}^2u} \mathop{Re }\bra Au|Bu\ket.\end{align}
\end{cor}
\begin{proof} Using Theorem \ref{51}, we have \begin{align*}\lim\limits_{t\rightarrow 0^+} \dfrac{\norm{A+t B} - \norm{A}}{t} =  \dfrac{1}{\norm{A}}\max\bigg\{\text{Re }\tr(A^*BT) & :  T \text{ is a positive element of } \mathscr K(\hc), \\
&  \norm{T}_1 = 1, \tr(A^*AT) = \norm{A}^2\bigg\}.\end{align*} Now using Lemma \ref{99}, we get  \begin{align*}\lim\limits_{t\rightarrow 0^+} \dfrac{\norm{A+t B} - \norm{A}}{t} =  \dfrac{1}{\norm{A}}\max\bigg\{\sum\limits_{j\in \mathcal J}s_j\bra A^*Bu_j|u_j\ket & :  s_j>0, \sum_js_j =1,\norm{u_j} = 1,\\
& A^*Au_j = \norm{A}^2u_j, \mathcal J \text{ is a countable set}\bigg\}.\end{align*}This gives the required result.\end{proof} 

The immediate consequence of Corollary \ref{71} is a characterization of smooth points of unit ball of $\mathscr B(\hc)$, which was first given in  \cite[Theorem 3.1]{Abatzoglou} (for $\bF = \bR$) and then an alternative proof was given in \cite[Corollary 3.3]{Keckic} (for $\bF = \bR$ or $\bC$). The proof we have given is simpler. 

\begin{cor}\label{78} An operator $A$  is a smooth point of the unit ball of $\mathscr B(\hc)$ if and only if $A$ attains its norm at $\pm h$ with $\norm{h} =1$ and $\sup\limits_{x\perp h, \norm{x} =1}\norm{Ax} <\norm{A}$. In that case, $\lim\limits_{t\rightarrow 0} \dfrac{\norm{A+t B} - \norm{A}}{t} = \mathop{Re }\bra Ah|Bh\ket$.
\end{cor}
\begin{proof} It was shown in Theorem 1 of \cite{Kittaneh}, if $\dist(A, \mathscr K(\hc)) = \norm{A}$, then $A$ is not a smooth point of the unit ball of $\mathscr B(\hc)$. It is well known fact that  if $A$ is a smooth point, then it can't attain its norm at more than one point (see \cite[Theorem 2.1]{Henn}). Hence we only need to prove that if $A$ attains its norm at $\pm h$ with $\norm{h} =1$ and $\sup\limits_{x\perp h, \norm{x} =1}\norm{Ax} <\norm{A}$, then $A$ is a smooth point. In this case, using Corollary \ref{71}, $$\lim\limits_{t\rightarrow 0^{+}} \dfrac{\norm{A+t B} - \norm{A}}{t} =   \dfrac{1}{\norm{A}}\text{ Re }\bra Ah|Bh\ket.$$ And we also have \begin{eqnarray*}\lim\limits_{t\rightarrow 0^{-}} \dfrac{\norm{A+t B} - \norm{A}}{t} &=& \lim\limits_{t\rightarrow 0^{-}} \dfrac{\norm{A-t(- B)} - \norm{A}}{t} = -\lim\limits_{t\rightarrow 0^{+}} \dfrac{\norm{A+t(- B)} - \norm{A}}{t} \\
&=& - \dfrac{1}{\norm{A}}\text{ Re }\bra Ah|-Bh\ket = \dfrac{1}{\norm{A}}\text{ Re }\bra Ah|Bh\ket.\end{eqnarray*}

Hence $\lim\limits_{t\rightarrow 0} \dfrac{\norm{A+t B} - \norm{A}}{t}$ exists and equal to $\text{ Re }\bra Ah|Bh\ket$. Since norm function is Gateaux differentiable at $A$, $A$ is a smooth point of the unit ball of $\mathscr B(\hc)$.
\end{proof}

The special case of the above theorem in the setting of $\mathscr{K}(\hc)$  for separable Hilbert space was first proved in \cite[Theorem 3.3]{Holub}. Using the above expression of Gateaux derivative and the fact that $\mathscr C_1(\hc)$ is the dual of $\mathscr{K}(\hc)$, we get that rank one operators are extreme points of $\mathscr C_1(\hc)$, along the lines of proof of \cite[Corollary 3.3]{Abatzoglou}. Characterization of the extreme points of $\mathscr C_1(\hc)$ was first proved in \cite[Theorem 3.1]{Holub}. Study of extreme points and smooth points  has been a subject of interest for many authors, see \cite{Abatzoglou, Heinrich,Holub, Kittaneh, Schatten}. Subdifferential sets in space of matrices $M_n(\bF)$ equipped with various norms have also been an interest to many authors, see \cite{2020} for a brief description. Using Corollary \ref{71}, we now give the following expression for the subdifferential set $\partial \|A\|$ when $\dist(A, \mathscr K(\hc))< \norm{A}$, with the idea used in \cite[Theorem 2]{Watson}.  

\begin{cor}\label{999}  Let $0\neq A\in \mathscr B(\hc)$ be such that $\dist(A, \mathscr K(\hc))< \norm{A}$, then \begin{equation*}\partial\norm{A} = \conv\{u\bar{\tens{}} v:\norm{u} = \norm{v} = 1, Au = \norm{A}v\}\end{equation*}\end{cor}
\begin{proof} Clearly  $\conv\{u\bar{\tens{}} v:\norm{u} = \norm{v} = 1, Au = \norm{A}v\}\subseteq\partial\norm{A}$. If equality doesn't occur, then there exists $T\in\partial\norm{A}\subseteq \mathscr C_1$ such that $T\notin\conv\{u\bar{\tens{}} v:\norm{u} = \norm{v} = 1, Au = \norm{A}v\}$. Then by the Hahn-Banach seperation theorem, there exists $B\in \mathscr K(\hc)$ such that $\tr(TB) = 1$ and $\tr(B(u\bar{\tens{}} v)) = \bra u|Bv\ket= 0$ for all $u.v\in\hc$ such that $\norm{u} = \norm{v} = 1, Au = \norm{A}v$.

Hence $\max\limits_{\norm{u} = \norm{v} = 1, Au = \norm{A}v} \text{ Re } \bra u|Bv\ket < \tr(BT) \leq \max\{\tr(BT) : T\in\partial\norm{A}\} = \lim\limits_{t\rightarrow 0^+} \dfrac{\norm{A+t B} - \norm{A}}{t},$

which contradicts Corollary \ref{71}.
\end{proof}

\section{Remarks}\label{remarks}

\begin{enumerate}
\item Let $A, B\in \mathscr B(\hc)$. Let $H_{\eps}$ stands for $E_{A^*A}((\norm{A}-\eps)^2, \norm{A}^2)$, and $E_{A^*A}$ stands for the spectral measure of operator $A^*A$. It was proved in \cite[Theorem 2.4]{Keckic} that \begin{equation}\lim\limits_{t\rightarrow 0^+} \dfrac{\norm{A+t B} - \norm{A}}{t} = \dfrac{1}{\norm{A}}\inf\limits_{\eps>0}\sup\limits_{h\in H_{\eps}, \norm{h} = 1} \text{Re} \bra Ah|Bh\ket.\label{9000}\end{equation} 
We will give an alternative proof of Theorem \ref{51} as an application of \eqref{9000}. Let $\alpha = \norm{A}\lim\limits_{t\rightarrow 0^{+}} \dfrac{\norm{A+t B} - \norm{A}}{t}$. Then we get a sequence $\{h_n\}$ such that $\norm{h_n} =1$, $h_n\in H_{1/n}$ and $\text{Re}\bra h_n|A^*Bh_n\ket\rightarrow \alpha$. Now $h_n\in H_{1/n}$ implies that $\norm{Ah_n}^2\in ((\norm{A}-1/n)^2, \norm{A}^2)$. We can choose a subsequence, if necessary, so that $\norm{Ah_n}\rightarrow \norm{A}$. Now define $\psi_n(T) = \bra h_n|Th_n\ket$. Then $\psi_n\in\mathcal S_{\mathscr B(H)}$ such that $\psi_n(A^*A)\rightarrow \norm{A}^2$ and $\psi_n(A^*B) \rightarrow \alpha$. Now using wea$k^*$-compactness of $\mathcal S_{\mathscr B(H)}$, there exists $\psi\in\mathcal S_{\mathscr B(H)}$ such that $\psi(A^*A) = \norm{A}^2$ and $\psi(A^*B) = \alpha$. By using the GNS construction, we get the proof of Theorem \ref{51}.

\item Using the Riesz Representation Theorem with Corollary \ref{68}, we get that if $f\in \C_b(\Omega)$ such that $\dist(f, \C_0(\Omega))< \norm{f}_{\infty}$, then for any $g\in \C_b(\Omega)$, we have
\begin{equation}\label{10000}\lim\limits_{t\rightarrow 0^+} \dfrac{\norm{f+t g}_{\infty} - \norm{f}_{\infty}}{t} = \sup\{\text{Re }(e^{-\iota arg f(x)}g(x)) : x\in\{z: |f(x)| = \norm{f}_{\infty}\}\}.\end{equation}
This is special case of \cite[Theorem 3.1]{Keckic 1} where it was proved that if $f,g\in \C_b(\Omega)$, then \begin{equation}\label{7000}\lim\limits_{t\rightarrow 0^{+}} \dfrac{\norm{f+t g} - \norm{f}}{t} = \inf\limits_{\delta>0}\sup\limits_{x\in E_{\delta}}\text{Re}(e^{-i\text{ arg } f(x)}g(x)),\end{equation} where $E_{\delta} = \{x\in \Omega : |f(x)| \geq\norm{f}_{\infty}-\delta\}$. Note that \eqref{10000} is a special case of \eqref{7000}. Similarly formula for Gateaux derivative obtained in \eqref{8000} is a special case of \eqref{9000}. Now equation \ref{10000} and formula \eqref{8000}, were obtained as the applications of Corollary \ref{68}. It raises question what can be analogous of Corollary \ref{68}, without the condition $\dist(f, \C_0(\Omega)) < \norm{f}_{\infty}$, that will be generalization of formulas \eqref{9000} and \eqref{7000}?

\end{enumerate}

\subsection*{Acknowledgements}

I would like to thank Amber Habib and Priyanka Grover, Shiv Nadar University for many useful discussions.


\begin{thebibliography}{10}

\bibitem{Abatzoglou} Abatzoglou, T. J. : Norm Derivatives on Spaces of Operators. \textit{Math. Ann.} 239(1979), 129-135.

\bibitem{Bhatia} Bhatia, R.; $\check{\text S}$emrl P. : Orthogonality of matrices and some distance problems. \textit{Linear Algebra Appl.} 287(1999), 77--85.

\bibitem{blackadar} Blackadar, B. : \textit{Operator Algebras - Theory of $C^*$-Algebras and von Neumann Algebras.} Springer-Verlag, Berlin 2006.

\bibitem{2020} Grover, P.; Singla, S. : \textit{Birkhoff-James orthogonality and applications : A survey.} Operator Theory, Functional Analysis and Applications, Birkh\"{a}user Basel, vol. 282, 2020.

\bibitem{2019} Grover, P.; Singla S. : Best approximations, distance formulas and orthogonality in $C^*$-algebras. {\it Communicated}.

\bibitem{Ideal} Harmand, P.; Werner, D.; Werner, W. : \textit{M-ideals in Banach spaces and Banach algebras.} Springer-Verlag, Berlin 1993. 

\bibitem{Heinrich} He\u{i}nrich, S. : The differentiability of the norm in spaces of operators. \textit{Functional Analysis and its Applications} 9(1975), 360--362.

\bibitem{Henn} Hennefeld, J. : Smooth, compact operators. {\it Proc. Amer. Math. Soc.} 77(1979), 87--90.

\bibitem{hiriart} Hiriart-Urruty, J. B. ; Lemar$\grave{\text e}$chal, C. : \textit{Fundamentals of Convex Analysis.} Springer, 2000.

\bibitem{Holub} Holub, J. R. : On the metric geometry of ideals of operators on Hilbert space. {\it Math. Ann.} 201(1973), 157--163

\bibitem{James}  James, R. C. : Orthogonality and linear functionals in normed linear spaces. \textit{Trans. Amer. Math. Soc.} 61(1947), 265--292.

\bibitem{Keckic} Ke$\check{\text c}$ki$\acute{\text c}$, D. J. : Gateaux Derivative of $B(H)$ Norm. {\it Proc. Amer. Math. Soc.} 133(2005), 2061--2067.

\bibitem{Keckic 1} Ke$\check{\text c}$ki$\acute{\text c}$, D. J. : Orthogonality and smooth points of unit ball of $C(K)$ and $C_b(\Omega)$. \textit{Eurasian Math. J.} 3(2012), 44--52.

\bibitem{Keckic 2} Ke$\check{\text c}$ki$\acute{\text c}$, D. J. : Orthogonality in ${\mathfrak S}_1$ and  ${\mathfrak S}_{\infty}$  spaces and normal derivations. {\it J. Operator Theory} 51(2004), 89--104.

\bibitem{Kittaneh} Kittaneh, F.; Younis, R. : Smooth points of certain operator spaces. {\it Integral Equ. Oper. Theory} 13(1990), 849--855.

\bibitem{arxiv} Mal, A.; Kallol, P. : Birkhoff-James orthogonality to as subspace of operators defined between Banach spaces. \url{https://arxiv.org/abs/1912.03635}

\bibitem{Sain} Paul, K.; Sain, D.; Ghosh, P. : Birkhoff-James orthogonality and smoothness of bounded linear operators. {\it Linear Algebra Appl.} 506 (2016), 551--563.

\bibitem{Schatten} Schatten, R. : The space of completely continuous operators on a Hilbert space. {\it Math. Ann.} 134(1957), 47--49.

\bibitem{Watson} Watson, G. A. : Characterization of the subdifferential of some matrix norms. {\it Linear Algebra Appl.} 170(1992), 33--45.

\bibitem{Wojcik} W$\acute{\text o}$jcik, P. : Birkhoff orthogonality in classical M-ideals. {\it J. Aust. Math. Soc.} 103(2017), 279--288.


\end{thebibliography}
\end{document}